\documentclass[a4paper, 12pt]{article}

\usepackage[utf8]{inputenc}
\usepackage[T1]{fontenc}
\usepackage[a4paper, margin=2.5cm]{geometry}
\usepackage{needspace}
\usepackage{cite}
\usepackage{amsmath, amsthm, amssymb}
\usepackage{graphicx}
\usepackage{enumitem}
\usepackage{authblk}
\usepackage[textsize=small, textwidth=2cm, color=yellow]{todonotes}
\usepackage{thmtools}
\usepackage{thm-restate}
\usepackage[colorlinks=true, linkcolor =blue!80, citecolor=orange!70!black]{hyperref}
\usepackage{float}

\usepackage[capitalize]{cleveref}

\Crefname{conjecture}{Conjecture}{Conjectures}
\usepackage{comment}

\renewenvironment{abstract}
{\small\vspace{-1em}
\begin{center}
\bfseries\abstractname\vspace{-.5em}\vspace{0pt}
\end{center}
\list{}{
\setlength{\leftmargin}{0.6in}%
\setlength{\rightmargin}{\leftmargin}}%
\item\relax}
{\endlist}

\declaretheorem[name=Theorem, numberwithin=section]{theorem}
\declaretheorem[name=Lemma, sibling=theorem]{lemma}
\declaretheorem[name=Proposition, sibling=theorem]{proposition}

\declaretheorem[name=Corollary, sibling=theorem]{corollary}
\declaretheorem[name=Conjecture, sibling=theorem]{conjecture}

\declaretheorem[name=Remark, style=remark, sibling=theorem]{remark}
\declaretheorem[name=Observation, style=remark, sibling=theorem]{observation}
\declaretheorem[name=Question, style=remark, sibling=theorem]{question}
\def\cqedsymbol{\ifmmode$\lrcorner$\else{\unskip\nobreak\hfil
\penalty50\hskip1em\null\nobreak\hfil$\lrcorner$
\parfillskip=0pt\finalhyphendemerits=0\endgraf}\fi}

\newcommand\eqdef{\overset{\text{\tiny{def}}}{=}}

\DeclareMathOperator{\tw}{\mathsf{tw}}

\newcommand{\N}{\mathbb{N}}
\newcommand{\NP}{\textsf{NP}}

\let\le\leqslant
\let\ge\geqslant
\let\leq\leqslant
\let\geq\geqslant

\newcommand{\tin}{\mathsf{tree} \textnormal{-} \alpha}
\newcommand{\tccn}{\mathsf{tree} \textnormal{-} \theta}
\newcommand{\fvs}{\mathsf{fvs}}
\newcommand{\ibn}{\mathsf{ibn}}

\title{Treewidth versus clique number. V. Further connections with tree-independence number}

\author[1]{Claire Hilaire}
\author[1,2]{Martin Milanič}
\author[1]{Đorđe Vasić}
\affil[1]{FAMNIT, University of Primorska, Koper, Slovenia}
\affil[2]{IAM, University of Primorska, Koper, Slovenia}

\date{}

\begin{document}

\maketitle

\begin{sloppypar}
\begin{abstract}
We continue the study of $(\tw,\omega)$-bounded graph classes, that is, hereditary graph classes in which large treewidth is witnessed by the presence of a large clique, and the relation of this property to boundedness of the tree-independence number, a graph parameter introduced independently by Yolov in 2018 and by Dallard, Milanič, and Štorgel in 2024.
Dallard et al.~showed that bounded tree-independence number is sufficient for $(\tw,\omega)$-boundedness, and conjectured that the converse holds.
While this conjecture has been recently disproved, it is still interesting to determine classes where the conjecture holds; for example, the conjecture is still open for graph classes excluding an induced star, as well as for finitely many forbidden induced subgraphs.
In this paper, we identify further families of graph classes where $(\tw,\omega)$-boundedness is equivalent to bounded tree-independence number.
We settle a number of cases of finitely many forbidden induced subgraphs, obtain several equivalent characterizations of $(\tw, \omega)$-boundedness in subclasses of the class of complements of line graphs, and give a short proof of a recent result of Ahn, Gollin, Huynh, and Kwon [SODA 2025] establishing bounded tree-independence number for graphs excluding a fixed induced star and a fixed number of independent cycles.


\medskip
\noindent{\bf MSC Classes (2020):}
05C75, % Structural characterization of families of graphs
05C05, % Trees
05C69, % Vertex subsets with special properties (dominating sets, independent sets, cliques, etc.)
05C83, % Graph minors
05C76. % Graph operations (line graphs, products, etc.)
\end{abstract}
\end{sloppypar}

\section{Introduction}

In this paper, we contribute to the study of $(\tw,\omega)$-bounded graph classes, continuing the research from~\cite{DBLP:journals/jctb/DallardMS24,zbMATH07436460
,dallard2023treewidth,dallard2024treewidth}.
The notion of $(\tw,\omega)$-boundedness is a relaxation of the notion of bounded treewidth that can also hold for dense graph classes.
A graph class is said to be \emph{$(\tw,\omega)$-bounded} if the only reason for large treewidth in graphs from the class is the presence of a large clique, and the same holds for induced subgraphs of graphs in the class.
More precisely, $(\tw,\omega)$-boundedness of a graph class $\mathcal{G}$ requires the existence of a function $f$ such that for every graph $G\in \mathcal{G}$ and every induced subgraph $G'$ of $G$, it holds that $\tw(G')\le f(\omega(G'))$, where $\tw$ and $\omega$ denote the treewidth and the clique number, respectively.

The $(\tw,\omega)$-boundedness property has some good algorithmic implications.
For instance, as shown by Chaplick et al.~\cite{MR4332111}, if a graph class $\mathcal{G}$ is $(\tw,\omega)$-bounded, then the \textsc{$k$-Clique} and \textsc{List $k$-Coloring} problems can be solved in polynomial time for graphs in $\mathcal{G}$.
Furthermore, in~\cite{zbMATH07436460} Dallard et al.~gave  improved approximations for the \textsc{Maximum Clique} problem and ask whether $(\tw,\omega)$-boundedness has any good algorithmic implications for problems related to independent set (see also~\cite{dallard2023treewidth}).
In particular, it is not known if the \textsc{Independent Set} problem is solvable in polynomial time in every $(\tw,\omega)$-bounded graph class.

In this context, Yolov~\cite{MR3775804}, and independently Dallard, Milani\v{c}, and \v{S}torgel~\cite{DBLP:journals/jctb/DallardMS24} defined a new parameter called tree-independence number, which is a generalization of treewidth, in which the goal is to have bags of small independence number instead of small size (see \cref{sec:prelim} for precise definitions).
Boundedness of the tree-independence number implies polynomial-time solvability for the \textsc{Independent Set} problem (see also~\cite{MR3775804,LMMORS} for other problems solvable in polynomial time in classes of graphs with bounded tree-independence number).
Furthermore, Dallard et al.~\cite{DBLP:journals/jctb/DallardMS24} showed that bounded tree-independence number is sufficient for $(\tw,\omega)$-boundedness, and conjectured that the converse holds.

\medskip
\begin{sloppypar}
\begin{conjecture}[Dallard et al.~\cite{dallard2023treewidth}]
\label{conj:tree-alpha}
Let $\mathcal{G}$ be a hereditary graph class.
Then $\mathcal{G}$ is $(\tw,\omega)$-bounded if and only if $\mathcal{G}$ has bounded tree-independence number.
\end{conjecture}
\end{sloppypar}

\Cref{conj:tree-alpha} was recently disproved by Chudnovsky and Trotignon (see~\cite{CT24}).
Nevertheless, it is still an interesting to ask in which graph classes the two properties of $(\tw,\omega)$-boundedness and bounded tree-independence number are equivalent.
Several recent works (see~\cite{MR4763243,arXiv:2406.13053,dallard2023treewidth}) established bounded tree-independence number for graph classes that were previously known to be $(\tw,\omega)$-bounded (see also~\cite{DBLP:conf/soda/ChudnovskyGHLS25,arXiv:2405.00265}).
Furthermore, the following result due to Dallard et al.~\cite{dallard2024treewidth} established \Cref{conj:tree-alpha} for subclasses of the class of line graphs.

\begin{theorem}[Dallard et al.~\cite{dallard2024treewidth}]\label{prop:line-graphs}
Let $\mathcal{G}$ be a class of graphs and let $L(\mathcal{G})$ be the class of line graphs of graphs in~$\mathcal{G}$.
Then, the following statements are equivalent.
\begin{enumerate}
\item The class $L(\mathcal{G})$ is $(\tw,\omega)$-bounded.
\item The class $L(\mathcal{G})$ has bounded tree-independence number.
\item The class $\mathcal{G}$ has bounded treewidth.
\end{enumerate}
\end{theorem}

\Cref{conj:tree-alpha} is still open for graph classes defined by finitely many forbidden induced subgraphs as well as for graph classes excluding some induced \emph{star}, that is, a complete bipartite graph of the form $K_{1,t}$ for some positive integer $t$.
Let us discuss these two cases in some more detail.

\subsection*{Finitely many forbidden induced subgraphs}
Lozin and Razgon~\cite{MR4385180} showed that if $\mathcal{G}$ is a class of graphs defined by finitely many forbidden induced subgraphs, then $\mathcal{G}$  has bounded treewidth if and only if it excludes at least one complete graph $K_t$, at least one complete bipartite graph $K_{t,t}$, at least one graph from the family $\mathcal{S}$ of graphs every component of which is a tree with at most three leaves, and at least one graph from the family $L(\mathcal{S})$ of line graphs of graphs in $\mathcal{S}$.
An immediate consequence of this result is the following corollary, where for a family $\mathcal{F}$ of graphs, a graph $G$ is said to be \emph{$\mathcal{F}$-free} if it does not contain an induced subgraph isomorphic to any member of $\mathcal{F}$.

\begin{corollary}\label{cor:LR}
For any finite family $\mathcal{F}$ of graphs, the class of $\mathcal{F}$-free graphs is $(\tw, \omega)$-bounded if and only if $\mathcal{F}$ contains a complete bipartite graph, a graph from $\mathcal{S}$, and a graph from~$L(\mathcal{S})$.
\end{corollary}

Since \Cref{conj:tree-alpha} is only meaningful for $(\tw,\omega)$-bounded graph classes, \Cref{cor:LR} implies that, when restricted to the case of finitely many forbidden induced subgraphs, \Cref{conj:tree-alpha} can be equivalently stated as follows.

\begin{conjecture}[Conjecture 1.2 in~\cite{dallard2024treewidth}]\label{conjecture:finitely-many-fis}
For any positive integer $t$ and any two graphs $S\in \mathcal{S}$ and $T\in L(\mathcal{S})$, the class of $\{K_{t,t},S,T\}$-free graphs has bounded tree-independence number.
\end{conjecture}

The conjecture is open already in the following special case (where $P_k$ denotes the $k$-vertex path).

\begin{conjecture}[Conjecture 1.3 in~\cite{dallard2024treewidth}]\label{conjecture:excluding-a-path}
For any two positive integers $k$ and $t$, the class of $\{P_k,K_{t,t}\}$-free graphs has bounded tree-independence number.
\end{conjecture}

Partial results on \Cref{conjecture:finitely-many-fis,conjecture:excluding-a-path} are known.
The result of Lozin and Razgon implies \cref{conjecture:finitely-many-fis} when restricted to classes of graphs with bounded maximum degree.
This was generalized by the main result of Dallard et al.~\cite{dallard2024treewidth}, establishing \Cref{conjecture:finitely-many-fis} for the case when $K_{t,t}$ is replaced with the star $K_{1,t}$.
A recent result of Abrishami et al.~\cite[Theorem 1.1]{abrishami2024excludingcliquebicliquegraphs} implies the validity of \Cref{conjecture:finitely-many-fis} for the case when $S = T = sP_2$ for some positive integer $s$.\footnote{The graph $sP_2$ is the disjoint union of $s$ copies of $P_2$.}
In~\cite{dallard2024treewidth}, Dallard et al.~proved \Cref{conjecture:excluding-a-path} for the classes of $\{P_4,K_{t,t}\}$-free graphs (any $t$) and $\{P_5,C_4\}$-free graphs (note that $C_4 = K_{2,2}$).
The conjecture is open for the classes of $\{P_5,K_{t,t}\}$-free graphs for all {$t>2$} as well as for the class of $\{P_6,C_4\}$-free graphs.
Recently, Chudnovsky et al.~(see~\cite[Theorem 1.2]{chudnovsky2025treeindependencenumberv} for a more general result) showed that for any two positive integers $k$ and $t$, the tree-independence number of a $\{P_k,K_{t,t}\}$-free graph is at most polylogarithmic in the number of vertices.

\subsection*{Excluding an induced star}

Korhonen showed in \cite{MR4539481} that, for graphs with bounded maximum degree (thus excluding some fixed star as a subgraph) an induced variant of the Grid-Minor Theorem~\cite{MR0854606} holds for treewidth: for any planar graph $H$, any class of graphs with bounded degree excluding $H$ as an induced minor has bounded treewidth.\footnote{For the definitions of the induced minor relation and $H$-induced-minor-free graphs, we refer to \cref{sec:prelim}.}
Actually, Korhonen stated the result only for the case when $H$ is a $(k\times k)$-grid; however, every planar graph is an induced minor of a sufficiently large grid graph (see~\cite[Theorem 12]{campbell2024treewidthhadwigernumberinduced}).

When restricted to graph classes excluding an induced star, Korhonen's theorem combined with Ramsey's theorem (see~\cite{MR1576401}) implies that, $(\tw,\omega)$-boundedness holds whenever some fixed planar graph is also excluded as an induced minor.
In fact, as shown by Dallard et al.~\cite{dallard2024treewidth}, in the setting of graph classes excluding an induced star, \Cref{conj:tree-alpha} is equivalent to the following conjecture~(\!\cite[Conjecture 7.2]{dallard2024treewidth}).

\begin{sloppypar}
\begin{conjecture}%[Dallard et al.~\cite{dallard2024treewidth}]
\label{conj:planar-H-no-star}
For every positive integer $t$ and every planar graph $H$, there exists an integer $N_{t,H}$ such that every $K_{1,t}$-free $H$-induced-minor-free graph $G$ satisfies \hbox{$\tin(G)\le N_{t,H}$}.
\end{conjecture}
\end{sloppypar}

\Cref{conj:planar-H-no-star} is known to hold for the following cases:
\begin{itemize}
    \item if $H\in \mathcal{S}$, as a consequence of the main result of Dallard et al.~\cite{dallard2024treewidth};
    \item if $H$ is the $4$-wheel (a $4$-cycle plus a universal vertex), $K_5$ minus an edge, or a graph of the form $K_{2,t}$ for some positive integer $t$, see~\cite{dallard2023treewidth} (in fact, in all these cases bounded tree-independence number holds even without excluding a star);
    \item if $H$ is the graph $kC_3$, the disjoint union of $k$ copies of the $3$-cycle, as shown recently by Ahn et al.~\cite{ahn2024coarseerdhosposatheorem};
    \item if $H$ is a cycle, as a consequence of a result by Seymour~\cite[Theorem 3.3]{MR3425243},
    and, even more generally, if $H$ is a wheel (that is, a cycle plus a universal vertex)~\cite{CHMW2025}.
\end{itemize}

\subsection*{Our results}

We identify further families of graph classes for which \Cref{conj:tree-alpha} holds.
We first observe that a result of Chudnovsky and Seymour on graphs with bounded splitness \cite{zbMATH06334633} leads to the resolution of \Cref{conj:tree-alpha} for hereditary graph classes that exclude some disjoint union of complete graphs.
As a consequence of this result and the fact that line graphs are claw-free, we establish a result analogous to \Cref{prop:line-graphs} for complements of line graphs.
With some additional arguments, we in fact obtain the following refined result, giving several equivalent characterizations of $(\tw,\omega)$-boundedness in this case.
For a graph class $\mathcal{G}$, we denote by $\overline{L(\mathcal{G})}$ the class of complements of line graphs of graphs in~$\mathcal{G}$.

\begin{restatable}{theorem}{thmColine}\label{thm:coline-graphs}
For every graph class $\mathcal{G}$, the following statements are equivalent.
\begin{enumerate}
\item\label{newproperty1} The class $\overline{L(\mathcal{G})}$ is $(\tw,\omega)$-bounded.
\item\label{newproperty5} The class $\overline{L(\mathcal{G})}$ has bounded tree-independence number.
\item\label{newproperty2} There exists a positive integer $s$ such that every graph in $\overline{L(\mathcal{G})}$ is $K_{s,s}$-free.
\item\label{newproperty4} There exists a positive integer $s$ such that every graph in $\overline{L(\mathcal{G})}$ has a vertex cover inducing a subgraph with independence number at most $s$.
\item\label{newproperty3} There exists a positive integer $s$ such that every graph in $\mathcal{G}$ is $2K_{1,s}$-subgraph-free.
\end{enumerate}
\end{restatable}

Before discussing our other results, let us put the result of \Cref{thm:coline-graphs} in perspective.
It can be observed that the forbidden induced subgraph characterization of the class of line graphs due to Beineke (see~\cite{zbMATH03321962}) implies that for any graph $G$, the complement of its line graph does not contain any two vertex-disjoint cycles $C$ and $C'$ that are \emph{independent}, in the sense that there is no edge between $C$ and $C'$ (see~\Cref{prop:co-line}).
In other words, for every graph $G$, the graph $\overline{L(G)}$ is $\mathcal{O}_2$-free, where a graph is said to be \emph{$\mathcal{O}_k$-free} if it does not contain $k$ pairwise independent cycles.
This motivates the question of whether the result of \Cref{thm:coline-graphs} can be generalized to classes of $\mathcal{O}_k$-free graphs for some $k\ge 2$, that is, whether for $\mathcal{O}_k$-free graphs, excluding some complete bipartite graph $K_{t,t}$ as an induced subgraph bounds the tree-independence number.
However, this is known not be the case, not even for $\mathcal{O}_2$-free graphs and $t = 2$ (see~\cite{MR4723425}).

What if instead of excluding a complete bipartite graph $K_{t,t}$ we exclude a star, $K_{1,t}$?
In this case, since a graph $G$ is $\mathcal{O}_k$-free if and only if $G$ is $kC_3$-induced-minor-free, the aforementioned result of Ahn et al.~\cite{ahn2024coarseerdhosposatheorem} proving \Cref{conj:planar-H-no-star} for the case when $H$ is the graph $kC_3$ shows that for any class of $\mathcal{O}_k$-free graphs, excluding a star does imply bounded tree-independence number.

\begin{theorem}[Ahn et al.~\cite{ahn2024coarseerdhosposatheorem}]\label{thm:K1t-free-Ok-free-tree-alpha-intro}
For every two integers $k\ge 1$ and $t\ge 2$ there exists an integer $N_{k,t}$ such that if $G$ is $K_{1,t}$-free $\mathcal{O}_k$-free graph, then $\tin (G) \le N_{k,t}$.
\end{theorem}

The proof of \cref{thm:K1t-free-Ok-free-tree-alpha-intro} is rather involved.
As our next result, we give a simpler proof of \cref{thm:K1t-free-Ok-free-tree-alpha-intro}, building on the work of Bonamy et al.~\cite{MR4723425}.
The bound given by our proof matches the dependency on $t$ with the bound given by Ahn et al.~\cite{ahn2024coarseerdhosposatheorem}, which is linear in $t$.
The dependency on $k$ in the bound given by Ahn et al.~\cite{ahn2024coarseerdhosposatheorem} in $k\log k$, while our proof gives a bound that is exponential in $k$.
So, while our approach gives suboptimal bounds, the proof is much shorter.

\medskip
Our last set of results deals with the case of finitely many forbidden induced subgraphs.
%We come closer to a positive resolution of the simplest open cases of  \Cref{conjecture:excluding-a-path}, namely the classes of $\{P_5,K_{t,t}\}$-free graphs (for $t>2$) and $\{P_6,C_4\}$-free graphs, by proving boundedness of tree-independence number for the class $\{P_3+P_1,K_{t,t}\}$-free graphs (for every positive integer $t$) and of the class of $\{P_4+P_1,C_4\}$-free graphs.
We come closer to a positive resolution of the simplest open cases of  \Cref{conjecture:excluding-a-path}, namely the classes of $\{P_5,K_{t,t}\}$-free graphs (for $t>2$) and $\{P_6,C_4\}$-free graphs.
The graph $P_5$ has three maximal proper induced subgraphs, namely $P_4$, $P_3+P_1$, and $2P_2$.
Recall that for \hbox{$H\in \{P_4,2P_2\}$}, it is already known that for every positive integer $t$, the class of $\{H,K_{t,t}\}$-free graphs has bounded tree-independence number (see~\cite{dallard2024treewidth} and \cite{abrishami2024excludingcliquebicliquegraphs}, respectively).
We prove boundedness of tree-independence number for the classes of $\{P_3+P_1,K_{t,t}\}$-free graphs, as well as of the class of $\{P_4+P_1,C_4\}$-free graphs.
In fact, for both cases we are able to prove something stronger.

For the former case we characterize the exact value of the tree-independence number of a $(P_3 + P_1)$-free graph $G$ in terms of the largest $t$ such that $K_{t,t}$ is an induced subgraph of $G$.
Following Dallard et al.~\cite{dallard2024treewidth}, who obtained a similar result for $P_4$-free graphs, we refer to this parameter as the \emph{induced biclique number} of $G$ and denote it by $\ibn(G)$.

\begin{restatable}{theorem}{thmPthreePone}\label{thm:P3+P1}
Let $G$ be a $( P_3 + P_1)$-free graph with at least one edge.
Then, $\tin(G) = \ibn(G)$, unless $G$ is $C_4$-free but contains an induced $C_5$, in which case $\tin(G) = 2$.
\end{restatable}
%\CH{Observe also that the results ...clique-partition-free... implies that $(kP_2,K_{t,t})$-free graphs (and in particular $(2P_2,K_{t,t})$-free) have bounded tree-independence number. Therefore,  \Cref{conjecture:excluding-a-path}}

In the latter case we prove boundedness of a parameter that is an upper bound for the tree-independence number, namely the so-called \emph{tree-clique-cover number}, denoted by $\tccn$, where, instead of bounding the size of an independent set in a bag of a tree decomposition, we bound the number of cliques needed to cover the vertices of a bag.

\begin{restatable}{theorem}{thmtccn}\label{thm:tccn}
If $G$ is a $\{P_4+P_1,C_4\}$-free graph, then $\tccn(G)\le 3$.
\end{restatable}

\noindent{\bf Structure of the paper.}
After collecting the basic definitions and results from the literature in \Cref{sec:prelim}, we discuss, in \Cref{sec:coline}, \Cref{conj:tree-alpha} for hereditary graph classes that exclude some disjoint union of complete graphs, as well as for complements of line graphs, proving in particular \Cref{thm:coline-graphs}.
In \Cref{sec:star}, we give a simpler proof of \cref{thm:K1t-free-Ok-free-tree-alpha-intro}.
In \Cref{sec:P3P1-free}, we prove \Cref{thm:P3+P1}.
In \Cref{sec:P4P1C4-free}, we prove \Cref{thm:tccn}.
We conclude the paper with some open problems in \Cref{sec:conclusion}.

\section{Preliminaries}\label{sec:prelim}

All the graphs considered in this paper are finite, simple (that is, without loops or parallel edges), and undirected.
An \emph{independent set} in a graph $G$ is a set of pairwise nonadjacent vertices.
The \emph{independence number} of a graph $G$, denoted $\alpha(G)$, is the size of a largest independent set in $G$.
A \emph{clique} in a graph $G$ is a set of vertices such that every two distinct vertices are adjacent.
The \emph{clique number} of a graph $G$, denoted by $\omega(G)$, is the size of a largest clique in $G$.
Given a set $S\subseteq V(G)$, we denote by $G[S]$ the subgraph of $G$ induced by $S$.
A \emph{vertex cover} in a graph $G$ is a set of vertices containing at least one endpoint of every edge of $G$.
For $v\in V(G)$, the set $N_G(v) = \{ u\in V(G) : vu\in E(G)\}$ is the \emph{neighborhood} of $v$,  and $N_G[v] = N(v) \cup \{v\}$ is the \emph{closed neighborhood} of $v$.
%\MM{the degree $d(v)$ needs to be defined, too}
The \emph{degree} of a vertex $v\in V(G)$, denoted by $d_G(v)$, is the number of edges incident with $v$, or equivalently, $d_G(v)=|N(v)|$.
If the graph is clear from the context, we may just write $N(v)$, $N[v]$, and $d(v)$ instead of $N_G(v)$, $N_G[v]$, and $d_G(v)$, respectively.
The graph isomorphism relation will be denoted by $\cong$.

Given a graph $G$ and an edge $uv$ in $G$, \emph{subdividing} the edge $uv$ means deleting the edge from the graph and adding a new vertex $w$ adjacent exactly to $u$ and $v$.
\emph{Contracting} an edge $e=uv$ in a graph $G$ means replacing $u$ and $v$ with a single new vertex $w$ so that the vertices adjacent to $w$ are exactly the vertices of $G$ that are adjacent to $u$ or $v$.

A graph class is \emph{hereditary} if for any graph $G$ in the class, every induced subgraph of $G$ also belongs to the class.
Given two graphs $G$ and $H$, we say that $G$ is \emph{$H$-free} if it does not contain an induced subgraph isomorphic to $H$.
Similarly, $G$ is \emph{$H$-subgraph-free} it does not contain a subgraph isomorphic to $H$.
For every set $S\subseteq V(G)$, we denote by $G-S$ the graph $G[V(G)\setminus S]$, and similarly, for every $u\in V(G)$, $G-u$ corresponds to the graph $G[V(G)\setminus \{u\}]$.
A graph $H$ is an \emph{induced minor} of a graph $G$ if a graph isomorphic to $H$ can be obtained from $G$ by a sequence of vertex deletions and edge contractions.
If a graph $G$ does not contain a graph $H$ as an induced minor, then we say that $G$ is \emph{$H$-induced-minor-free}.

A \emph{forest} is an acyclic graph (i.e., a graph without any cycles).
The \emph{girth} of a graph $G$ is the minimum length of a cycle in $G$ (and $\infty$ if $G$ is a forest).
A \emph{feedback vertex set} $X$ in a graph $G$ is a set of vertices of $G$ such that $G-X$ is a forest.
The minimum size of a feedback vertex set in $G$ is denoted by $\fvs (G)$.
The \emph{cycle rank} of a graph $G$ is defined as $r(G)= |E(G)| - |V(G)| + |C(G)|$ where $C(G)$ denotes the set of connected components of $G$.
The cycle rank is exactly the number of edges of $G$ which must be removed to make $G$ a forest.

A \emph{complete graph} of order $n$ is denoted by $K_n$ (in particular, $K_3$ is called a \emph{triangle}), and a complete graph of order $n$ with a single edge removed is denoted by $K_n^-$.
A \emph{complete multipartite graph} is any graph $G$ such that there exists a partition of its vertex set into $k\ge 1$ parts such that two distinct vertices are adjacent in $G$ if and only if they belong to different parts.
Given two integers $m,n\ge 0$, the \emph{complete bipartite graph} $K_{m,n}$ is a complete multipartite graph with exactly two parts, one of size $m$ and one of size $n$.
When $m = n$, we speak about \emph{balanced} complete bipartite graphs.
A \emph{claw} is a complete bipartite graph $K_{1,3}$.
The path graph and the cycle graph with $n$ vertices are denoted by $P_n$ and $C_n$, respectively.
Given a graph $G$ and an integer $k\ge 0$, a \emph{path of length $k$ in} $G$ is a sequence $P = v_1\ldots v_{k+1}$ of pairwise distinct vertices of $G$ such that $v_iv_{i+1}\in E(G)$ for all $i = 1,\ldots, k$.
The vertices $v_1,v_{k+1}$ are called the \emph{extremities} of $P$.
Vertices of $P$ that are not its extremities are the \emph{internal} vertices of $P$.
A path $v_1\ldots v_k$ with $k\ge 3$ and $v_kv_1\in E(G)$ is a \emph{cycle of length $k$ in} $G$.
An edge $v_iv_j\in E(G)$ with $|i-j|>1$ is a \emph{chord} of the path; the path is said to be \emph{chordless} if it has no chords.
Chordless cycles are defined similarly.
A \emph{chordal graph} is a graph in which all chordless cycles have length $3$.

A \emph{tree decomposition} of a graph $G$ is a pair $\mathcal{T}=(T, \{ X_t \}_{t\in V(T)})$ where $T$ is a tree and every node $t$ of $T$ is assigned a vertex subset $X_t\subseteq V(G)$ called a \emph{bag} such that the following conditions are satisfied: every vertex $v$ is in at least one bag, for every edge $uv\in E(G)$ there exists a node $t\in V(T)$ such that $X_t$ contains both $u$ and $v$, and for every vertex $u\in V(G)$ the subgraph $T_u$ of $T$ is induced by the set $\{ t\in V(T) : u\in X_t \}$ is connected, (that is, a tree).
The \emph{width} of $\mathcal{T}$, denoted by $width(\mathcal{T})$, is the maximum value of $|X_t|-1$ over all $t\in V(T)$.
The \emph{treewidth} of a graph $G$, denoted by $\tw(G)$, is defined as the minimum width of a tree decomposition of $G$.
The \emph{independence number} of $\mathcal{T}$ is defined as
\[\alpha(\mathcal{T}) = \max_{t\in V(T)} \alpha(G[X_t])\,.\]
The \emph{tree-independence number} of $G$, denoted by $\tin (G)$, is the minimum independence number among all possible tree decompositions of $G$.
A graph class $\mathcal{G}$ is said to be \emph{$(\tw, \omega)$-bounded} if there exists a function $f:\N \rightarrow \N$ such that for every graph $G\in \mathcal{G}$ and every induced subgraph $G'$ of $G$, we have $\tw(G')\le f(\omega (G'))$.

\begin{sloppypar}
We recall two results from~\cite{DBLP:journals/jctb/DallardMS24,zbMATH07436460}, giving a sufficient and a necessary condition for $(\tw, \omega)$-boundedness, respectively.
The first one relates $(\tw, \omega)$-boundedness to bounded tree-independence number (see \cite[Lemma 3.2]{DBLP:journals/jctb/DallardMS24}).
\end{sloppypar}

\begin{lemma}\label{lem:bdd-tree-alpha-tw-omega-bdd}
Every graph class with bounded tree-independence number is $(\tw, \omega)$-bounded.
\end{lemma}

The second gives a particular family of obstructions to $(\tw, \omega)$-boundedness (see, e.g.,~\cite[Lemma 2.7]{zbMATH07436460}).

\begin{lemma}\label{lem:cbp}
The class of balanced complete bipartite graphs is not $(\tw, \omega)$-bounded.
\end{lemma}

\begin{sloppypar}
Dallard et al.~gave the following characterization of chordal graphs in terms of tree-independence number.
\begin{lemma}[Theorem 3.3 in \cite{DBLP:journals/jctb/DallardMS24}]\label{lem chordal tin}
    A graph $G$ is chordal if and only if $\tin(G)\leq 1$.
\end{lemma}
\end{sloppypar}

Given two vertex-disjoint graphs $G$ and $H$, their \emph{disjoint union} is the graph $G+H$ with vertex set $V(G)\cup V(H)$ and edge set $E(G)\cup E(H)$.
Given a graph $G$, the \emph{complement} $\overline{G}$ is the graph with vertex set $V(G)$ in which two distinct vertices are adjacent if and only if they are nonadjacent in $G$.
Given a graph $G$, its \emph{line graph} $L(G)$ is a graph such that each vertex of $L(G)$ is an edge of $G$ and two vertices of $L(G)$ are adjacent if and only if their corresponding edges share a vertex in $G$.

In 1970, Beineke characterized the class of line graphs by a set of nine forbidden induced subgraphs $G_1,\ldots, G_9$ (see~\cite{zbMATH03321962}).
Since we will not need all of these nine graphs for our purpose, let us only note that considering only three of the forbidden induced subgraphs (namely, $G_1$, $G_3$, and $G_6$ from~\cite[Figure 3]{zbMATH03321962}), the characterization implies the following.

\begin{lemma}
Let $H$ be a graph and let $G$ be the line graph of $H$.
Then $G$ is $\{$claw, $K_5^-$, $\overline{C_4+2K_1}\}$-free.
\end{lemma}

\begin{corollary}\label{cor:line-complement}
Let $H$ be a graph and let $G$ be the complement of the line graph of $H$.
Then $G$ is $\{K_3+K_1$, $K_2+3K_1$, $C_4+2K_1\}$-free.
\end{corollary}

\section{Excluding a clique partition graph}\label{sec:coline}

Chudnovsky and Seymour define in~\cite{zbMATH06334633} a graph $G$ to be \emph{$k$-split} if $V(G)$ can be partitioned into sets $A$ and $B$ such that $\omega(G[A])\le k$ and $\alpha(G[B])\le k$.
The \emph{splitness} of a graph $G$ is the minimum $k$ such that $G$ is $k$-split.
The following result completely characterizes minimal hereditary graph classes with unbounded splitness.
A \emph{clique partition graph} is a disjoint union of complete graphs, or, equivalently, any graph whose complement is complete multipartite.

\begin{theorem}[Chudnovsky and Seymour \cite{zbMATH06334633}]\label{chudnovsky}
For every clique partition graph $H_1$ and complete multipartite graph $H_2$, there exists an integer $k$ such that every $\{ H_1,H_2\}$-free graph is $k$-split.
\end{theorem}

\Cref{chudnovsky} implies the validity of \Cref{conj:tree-alpha} for graph classes excluding some clique partition graph.

\begin{theorem}\label{thm:cpg}
Let $\mathcal{G}$ be a hereditary graph class that does not contain all clique partition graphs.
Then, $\mathcal{G}$ is $(\tw,\omega)$-bounded if and only if $\mathcal{G}$ has bounded tree-independence number.
\end{theorem}

\begin{proof}
Since, by \cref{lem:bdd-tree-alpha-tw-omega-bdd}, any graph class with bounded tree-independence number is $(\tw, \omega)$-bounded, it suffices to prove the forward implication.

Suppose that $\mathcal{G}$ is a $(\tw,\omega)$-bounded hereditary  graph class that does not contain all clique partition graphs.
By \Cref{lem:cbp}, $\mathcal{G}$ does not contain all complete bipartite graphs.
Let $H_1$ be a clique partition graph and let $H_2$ be a complete bipartite graph such that $H_1$ and $H_2$ do not belong to $\mathcal{G}$.
Then, every graph in $\mathcal{G}$ is $\{H_1,H_2\}$-free.
Hence, by \Cref{chudnovsky}, there exists an integer $k$ such that every graph in $\mathcal{G}$ is $k$-split.
Let $G$ be a graph from $\mathcal{G}$.
Partition $V(G)$ into two sets $A$ and $B$ such that $A$ has clique number at most $k$ and $B$ has independence number at most $k$.
Since $\mathcal{G}$ is hereditary and $(\tw,\omega)$-bounded, $G[A]$ has bounded treewidth.
Let $\mathcal{T}_A$ be a tree decomposition of the graph $G[A]$ minimizing the width.
Then, adding the vertices in $B$ to every bag of $\mathcal{T}_A$ results in a tree decomposition $\mathcal{T}$ of $G$ with bounded independence number.
This shows that $\mathcal{G}$ has bounded tree-independence number, as claimed.
\end{proof}

\subsection*{Complements of line graphs}

\Cref{thm:cpg} has the following consequence for  complements of line graphs.

\begin{corollary}\label{cor:coline}
For every graph class $\mathcal{G}$, the class $\overline{L(\mathcal{G})}$ is $(\tw,\omega)$-bounded if and only if it has bounded tree-independence number.
\end{corollary}

\begin{proof}
By \cref{lem:bdd-tree-alpha-tw-omega-bdd}, it suffices to prove the forward implication.
Assume that $\overline{L(\mathcal{G})}$ is $(\tw,\omega)$-bounded.
Let $\mathcal{H}$ be the hereditary closure of $\overline{L(\mathcal{G})}$, that is, the class of all induced subgraphs of graphs in $\overline{L(\mathcal{G})}$.
By \Cref{cor:line-complement}, every graph in $\overline{L(\mathcal{G})}$ is $(K_3+K_1)$-free, and, consequently, so is every graph in $\mathcal{H}$.
Since $K_3+K_1$ is a clique partition graph, \Cref{thm:cpg} applies, showing that $\mathcal{H}$ has bounded tree-independence number.
Hence, so does its subclass $\overline{L(\mathcal{G})}$.
\end{proof}

We now refine the result of \Cref{cor:coline}, by showing that within subclasses of the class of complements of line graphs, $(\tw, \omega)$-boundedness can be characterized in several  equivalent ways.
This includes the exclusion of some balanced complete bipartite graph, the existence of a vertex cover with small independence number, and an excluded subgraph condition on the root graph (which is easily seen to be equivalent to the condition that the root graph contains at most one vertex of large degree).
For convenience, we restate the characterizations.

\thmColine*

\begin{proof}
First, \Cref{cor:coline} establishes the equivalence between Conditions~\ref{newproperty1} and~\ref{newproperty5}.

Second, we prove that Condition~\ref{newproperty1} implies Condition~\ref{newproperty2} by contraposition.
Assume that for every positive integer $s$ there is a graph $G\in \overline{L(\mathcal{G})}$ that contains an induced subgraph isomorphic to $K_{s,s}$.
Then, by \cref{lem:cbp}, $\overline{L(\mathcal{G})}$ is not $(\tw, \omega)$-bounded.

Next, we prove that Condition~\ref{newproperty2} implies Condition~\ref{newproperty3}.
Assume that there exists a positive integer $s$ such that every graph in $\overline{L(\mathcal{G})}$ is $K_{s,s}$-free.
We claim that every graph in $\mathcal{G}$ is $2K_{1,s}$-subgraph-free.
Suppose for a contradiction that there is a graph $G$ in $\mathcal{G}$ that has a subgraph $H$ isomorphic to $2K_{1,s}$.
Then $\overline{L(H)}\cong \overline{L(2K_{1,s})}\cong \overline{2K_s}\cong K_{s,s}$ since $L(K_{1,s}) \cong K_s$.
Since $H$ is a subgraph of $G$, it follows that $L(H)$ is an induced subgraph of $L(G)$, hence, $\overline{L(H)}\cong K_{s,s}$ is an induced subgraph of $\overline{L(G)}$, a contradiction.

Next, we prove that Condition~\ref{newproperty3} implies Condition~\ref{newproperty4}.
Suppose that $s\ge 4$ is an integer such that every graph in $\mathcal{G}$ is $2K_{1,s}$-subgraph-free.
We first show that every graph in $\mathcal{G}$ has at most one vertex of degree greater than $2s$.
Let $G$ be a graph in $\mathcal{G}$ and suppose that there is a vertex $v$ of $G$ such that $d(v)\ge 2s+1$.
We claim that every vertex $w\neq v$ has degree at most~$s$.
Suppose for a contradiction that there exists a vertex $w\neq v$ with degree at least $s+1$.
Fix $s$~vertices $w_1,\ldots, w_s$ from the set $N(w)\setminus\{v\}$ and $s$ vertices $v_1,\ldots, v_s$ from the set $N(v)\setminus\{w,w_1,\ldots, w_s\}$.
Then the subgraph of $G$ formed by the vertices $\{v,v_1,\ldots, v_s,w,w_1,\ldots, w_s\}$ and edges $\{ww_i\colon 1\le i\le s\}\cup \{vv_i\colon 1\le i\le s\}$ is isomorphic to $2K_{1,s}$, a contradiction.
This shows the claim that every graph in $\mathcal{G}$ has at most one vertex of degree greater than $2s$.

Now we prove that every graph in $\overline{L(\mathcal{G})}$ has a vertex cover inducing a subgraph with independence number at most $2s$.
Let $G$ be a graph in $\mathcal{G}$ and let $G' = \overline{L(G)}$.
By the above observation, we may assume that $G$ has a vertex, say $v$, such that the graph $G-v$ has maximum degree at most $2s$, and let $S=E(G-v)$.
Note that the set of edges of $G$ incident to $v$ forms a clique in $L(G)$ and hence an independent set in $G'$.
Consequently, the set $S$ forms a vertex cover in $G'$.
Furthermore, since a maximum independent set in $G'$ corresponds to a maximum clique in $L(G)$, it follows that the independence number of the subgraph of $G'$ induced by $S$ equals the clique number of the subgraph of $L(G)$ induced by $S$.
By the definition of $S$, the subgraph of $L(G)$ induced by $S$ is exactly the line graph of $G-v$.
Since $G-v$ has maximum degree at most $2s$, the clique number of $L(G-v)$ is at most $2s$.
We infer that the independence number of the subgraph of $G'$ induced by $S$ is at most $2s$.

%4) Condition~\ref{property6} implies Condition~\ref{newproperty2}:

Finally, we prove that Condition~\ref{newproperty4} implies Condition~\ref{newproperty5}.
Suppose that there exists an integer $s$ such that every graph in $\overline{L(\mathcal{G})}$ has a vertex cover inducing a subgraph with independence number at most $s$.
We claim that the tree-independence number of any graph in $\overline{L(\mathcal{G})}$ is at most $s+1$.
Let $G$ be a graph in $\mathcal{G}$ and let $S$ be a vertex cover in $G$ such that $\alpha(G[S])\le s$.
We have to show that there exists a tree decomposition $\mathcal{T}$ of $G$ with independence number at most $s+1$.
Recall that by definition of $S$, $G-S$ is an edgeless graph.
We form $\mathcal{T} = (T,\{ X_t \}_{t\in V(T)})$ in the following way: let $T$ be any tree with vertex set $V(G)\setminus S$ and, for every vertex $t\in V(T)$, define the bag $X_t=S\cup \{t\}$.
By construction, the union of all bags of $\mathcal{T}$ equals the vertex set of $G$.
Moreover, since $S$ is a vertex cover in $G$, for every edge $uv\in E(G)$ there exists a node $t\in V(T)$ such that $X_t$ contains both $u$ and $v$.
Furthermore, for every vertex $u\in V(G)$, the subgraph $T_u$ of $T$ induced by the set $\{ t\in V(T) : u\in X_t \}$ is either $T$ (if $u\in S$) or a one-vertex subtree of $T$ (otherwise); hence, $T_u$ is connected.
Thus, $\mathcal{T}$ is indeed a tree decomposition of $G$.
Since the independence number of the subgraph of $G$ induced by $S$ is at most $s$, it follows that the independence number of the subgraph of $G$ induced by $X_t$ is at most $s+1$, for any $t\in V(T)$.
Hence, the tree-independence number of $G$ is at most $s+1$ and therefore the class $\overline{L(\mathcal{G})}$ has bounded tree-independence number, since $G$ was an arbitrary graph from $\overline{L(\mathcal{G})}$.
\end{proof}

Let us also prove the result mentioned in the introduction, that no complement of a line graph contains two independent cycles.

\begin{proposition}\label{prop:co-line}
For every graph $G$, the graph $\overline{L(G)}$ is $\mathcal{O}_2$-free.
\end{proposition}

\begin{proof}
Let $G$ be a graph and suppose for a contradiction that $\overline{L(G)}$ is not $\mathcal{O}_2$-free.
Then, $\overline{L(G)}$ has two induced independent cycles.
Let $C$ and $D$ be two such cycles in $\overline{L(G)}$, and let $k$ and $\ell$ be their respective lengths.

We first show that $k\ge 4$.
Suppose for a contradiction that $k = 3$.
Let $V(C)= \{ v_1,v_2,v_3 \}$ and let $u$ be an arbitrary vertex of $D$.
Due to the independence of $C$ and $D$, the vertex $u$ is not adjacent to any of $v_1$, $v_2$ and $v_3$ in $\overline{L(G)}$.
Hence, the subgraph of $\overline{L(G)}$ induced by $\{v_1, v_2,v_3,u\}$ is isomorphic to $K_3+K_1$, a contradiction with \Cref{cor:line-complement}.
Therefore $k\ge 4$ and, by symmetry, $\ell\ge 4$.

Since $D$ has length at least $4$, there exist a pair of nonadjacent vertices in $D$, say $u_1$ and $u_2$.

Next, we show that $k = 4$.
Suppose for a contradiction that $k\ge 5$.
Then, there is a vertex $v_1\in V(C)$ and an edge $v_2v_3\in E(C)$ such that $v_1$ is at distance at least two from both $v_2$ and $v_3$.
Now it can be observed that, since $C$ and $D$ are independent cycles, the subgraph of $\overline{L(G)}$ induced by $\{ v_1, v_2, v_3, u_1, u_2 \}$ is isomorphic to $K_2 + 3K_1$, contradicting \cref{cor:line-complement}.

Since $C\cong C_4$, and $C$ and $D$ are independent cycles, the subgraph of $\overline{L(G)}$ induced by $V(C)\cup \{u_1,u_2\}$ is isomorphic to $C_4+2K_1$, which is again in contradiction with \cref{cor:line-complement}.
This completes the proof.
\end{proof}

Note that \Cref{prop:co-line} implies that any problem that is \NP-hard in the class of complements of line graphs is also \NP-hard in the class of $\mathcal{O}_2$-free graphs.
This is the case, for instance, for the \textsc{$k$-Clique Cover} problem for all $k\ge 3$.
This problem asks, given a graph $G$, whether $G$ contain $k$ cliques with union $V(G)$.
Indeed, for every $G$ and $k\geq 3$, determining whether $\overline{L(G)}$ can be covered with $k$ cliques is equivalent to determining whether $G$ is $k$-edge-colorable, which is known to be \NP-complete (see~\cite{H81}).

\section{Excluding a star and many independent cycles}\label{sec:star}

In this section, we study $\mathcal{O}_k$-free graphs excluding an induced star, giving a simpler proof of \cref{thm:K1t-free-Ok-free-tree-alpha-intro}.
Note that $K_{1,1}$-free graphs are precisely the edgeless graphs, so when studying $K_{1,t}$-free graphs, we focus on the case $t\geq 2$.
We first recall two theorems from Bonamy et al.~\cite[Theorems 5.5 and 5.3]{MR4723425}.
The bounds on $\delta_k$ and $\varepsilon_k$ that we state here are not explicit in their theorem statements, but implicit in the proofs.

\begin{theorem}[Bonamy et al.~\cite{MR4723425}]\label{thm:delta-k}
For any integer $k\ge 2$, there is some $\delta_k = \Omega\left(\frac{1}{k^6\log^3k}\right)$ such that if $G$ is a connected $\mathcal{O}_k$-free graph with girth at least $11$, and furthermore $G$ admits a shortest cycle $C$ such that $G-N[V(C)]$ is a forest, then $G$ has a vertex of degree at least $\delta_k \cdot r(G)$.
\end{theorem}

\begin{theorem}[Bonamy et al.~\cite{MR4723425}]\label{thm:fvsLog}
For any integer $k\ge 2$, there is some $\varepsilon_k = \Omega\left(\frac{1}{20^{k}}\right)$ such that any $\mathcal{O}_k$-free graph $G$ with girth at least 11 has a vertex of degree at least $\varepsilon_k \cdot r(G)$.
\end{theorem}

The bound on $\varepsilon_k$ given in~\Cref{thm:fvsLog} follows from the inductive proof of~\cite[Theorem 5.3]{MR4723425}, where $\varepsilon_k$ is defined as follows (using $\delta_k$ from~\Cref{thm:delta-k}):
$\varepsilon_2 = \delta_2$ and for every $k\ge 3$, the value of $\varepsilon_k$ is defined with the following formula:
\begin{equation}\label{eq:epsilon-k}
    \varepsilon_k \eqdef \min\left\{
      \frac{\varepsilon_{k-1}}{20},
      \frac{\delta_k}{20},
      \frac{\delta_k}{5(k+1)},
      \frac{1}{30(k-2)}
    \right\}.
\end{equation}
From here and the fact that $\delta_k = \Omega\left(\frac{1}{k^6\log^3k}\right)$, we infer that $\varepsilon_k = \Omega\left(\frac{1}{20^{k}}\right)$, as claimed.

\Cref{thm:fvsLog} implies that every $K_{1,t}$-free $\mathcal{O}_k$-free graph with sufficiently large girth has a small feedback vertex set.

\begin{lemma}\label{lem:fvs}
For every two integers $k\ge 1$ and $t\ge 2$ there exists an integer $c_{k,t} = \mathcal{O}\left(20^kt\right)$ such that if $G$ is a $K_{1,t}$-free $\mathcal{O}_k$-free graph with girth at least 11, then $\fvs (G)\le c_{k,t}$.
\end{lemma}

\begin{proof}
If a graph $G$ is $\mathcal{O}_1$-free, then $G$ is a forest so $\fvs(G)=0$, and thus we can take $c_{1,t}=0$ for every $t\geq 2$.

Suppose now that $k\ge 2$.
Fix $t\ge 2$ and let $G$ be a $K_{1,t}$-free $\mathcal{O}_k$-free graph with girth at least 11.
Let $\varepsilon_k$ be the constant given by \cref{thm:fvsLog}.
By \cref{thm:fvsLog}, $G$ has a vertex $v\in V(G)$ such that $d(v)\ge \varepsilon_k \cdot r(G)$.
Notice that $N(v)$ is an independent set, since otherwise $G$ would contain a cycle of length 3, contradicting the fact that $G$ has girth at least 11.
Hence, since $G$ is $K_{1,t}$-free, the size of $N(v)$ is at most $t-1$.
It follows that $t-1\ge d(v)$ and this gives us $\frac{t-1}{\varepsilon_k}\ge r(G)$.
Since the cycle rank $r(G)$ is the smallest number of edges of $G$ that we must remove from $G$ to make it a forest, $r(G)$ is also an upper bound on the size of a minimum feedback vertex set, i.e., $\fvs (G)\le r(G)$.
This shows that $\fvs (G)\le\lfloor \frac{t-1}{\varepsilon_k}\rfloor$.
Letting $c_{k,t}=\lfloor \frac{t-1}{\varepsilon_k}\rfloor$, gives us the desired inequality $\fvs (G)\le c_{k,t}$.
Note that $c_{k,t} = \mathcal{O}\left(20^kt\right)$,
since $\varepsilon_k = \Omega\left(\frac{1}{20^{k}}\right)$.
\end{proof}

\begin{remark}\label{rem:ckt-nondecreasing}
Observe that, for every $k\ge 1$ and $t\ge 2$, we can choose $c_{k,t}$ to be the smallest integer satisfying the condition of \Cref{lem:fvs}.
Therefore, for every $t\geq 2$, the sequence $(c_{k,t})_{k\ge 1}$ can be assumed to be nondecreasing, that is, $c_{k,t}\ge c_{k-1,t}$ for all $k,t\ge 2$.
Similarly, for every $k\geq 1$, the sequence $(c_{k,t})_{t\ge 2}$ can be assumed to be nondecreasing (although we will not need this fact).
\end{remark}

\begin{lemma}\label{lem: bddfvs}
For every two integers $k\ge 1$ and $t\ge 2$, every $K_{1,t}$-free $\mathcal{O}_k$-free graph $G$ contains two sets $S_1,S_2\subseteq V(G)$ such that:
\begin{itemize}
\item $|S_1|\le c_{k,t}$, where $c_{k,t}$ is the integer given by \Cref{lem:fvs},
\item $S_2$ induces a subgraph with independence number at most $10(k-1)(t-1)$, and
\item $S_1\cup S_2$ is a feedback vertex set in $G$.
\end{itemize}
\end{lemma}

\begin{proof}
Let us verify that the stated properties hold by induction on $k$.
If $k=1$, then $G$ is acyclic, so we can take $S_1 = S_2 = \emptyset$.

Let $k\ge 2$ and let $G$ be a $K_{1,t}$-free $\mathcal{O}_k$-free graph for some fixed $t\ge 2$.
If the girth of $G$ is at least 11, by \Cref{lem:fvs}, $G$ admits a feedback vertex set  $S_1$ with $|S_1|\le c_{k,t}$.
Such a set $S_1$ together with $S_2=\emptyset$ yields a pair of sets as required.

Now suppose that the girth of $G$ is at most 10.
Let $C$ be a shortest cycle in $G$, let $N$ be the neighborhood of $V(C)$, and let $R=V(G) \setminus (V(C)\cup N)$.
Then $|V(C)|\le 10$.
Notice that there are no edges between the vertices in $C$ and $R$, therefore $G[R]$ is a $K_{1,t}$-free $\mathcal{O}_{k-1}$-free graph.
By the induction hypothesis, $G[R]$ contains two sets $R_1, R_2\subseteq R$ such that $|R_1|\le c_{k-1,t}$, the set $R_2$ induces a subgraph with independence number at most $10(k-2)(t-1)$, and $R_1\cup R_2$ is a feedback vertex set in $G[R]$.
Notice also that for every $v\in V(C)$, the independence number of $G[N[v]]$ is at most $t-1$, since $G$ is $K_{1,t}$-free.
It follows that $\alpha(G[V(C)\cup N])\le 10(t-1)$.
Define $S_1 = R_1$ and $S_2 = V(C)\cup N\cup R_2$.
Since $R_1\cup R_2$ is a feedback vertex set in $G[R]$, $S_1\cup S_2$ is a feedback vertex set in $G$.
Furthermore, \[\alpha(G[S_2]) \le 10(t-1)+10(k-2)(t-1)=10(k-1)(t-1)\] and $|S_1| = |R_1|\le c_{k-1,t} \le c_{k,t}$ (where the last inequality is justified by \Cref{rem:ckt-nondecreasing}).
This completes the induction step and with it the proof of the theorem.
\end{proof}

\begin{sloppypar}
\cref{lem: bddfvs} implies that the tree-independence number of $K_{1,t}$-free $\mathcal{O}_k$-free graphs is bounded, leading to the following simpler proof of \cref{thm:K1t-free-Ok-free-tree-alpha-intro}.
\end{sloppypar}

\begin{proof}[Proof of \cref{thm:K1t-free-Ok-free-tree-alpha-intro}.]
Fix two integers $k\ge 1$ and $t\ge 2$, and let $G$ be a $K_{1,t}$-free $\mathcal{O}_k$-free graph.
By \cref{lem: bddfvs}, there exist two sets $S_1,S_2\subseteq V(G)$ such that $|S_1|\le c_{k,t}$, the set $S_2$ induces a subgraph of $G$ with independence number at most $10(k-1)(t-1)$, and $S_1\cup S_2$ is a feedback vertex set in $G$, where $c_{k,t}$ is the integer from \cref{lem:fvs}.

Let $G' = G \setminus (S_1 \cup S_2)$.
Note that $G'$ is acyclic, hence chordal, and therefore by \cref{lem chordal tin}, \hbox{$\tin (G')\le 1$}.
Let $\mathcal{T}'$ be a tree decomposition of $G'$ with independence number at most 1.
By adding the vertices in $S_1\cup S_2$ to every bag of $\mathcal{T}'$, we obtain a tree decomposition $\mathcal{T}$ of $G$.

Since $\alpha(G[S_1\cup S_2])\le c_{k,t}+10(t-1)(k-1)$, the independence number of every bag of $\mathcal{T}$ is at most $c_{k,t}+10(t-1)(k-1) + 1$.
Setting $N_{k,t} = c_{k,t} + 10(t-1)(k-1) + 1$, this gives us that $\tin (G)\le N_{k,t}$.
Since  $c_{k,t}= \mathcal{O}\left(20^kt\right)$, the same bound holds also for $N_{k,t}$.
\end{proof}

We conclude this section by explaining why the statement of  \cref{thm:K1t-free-Ok-free-tree-alpha-intro} cannot be generalized to $K_{t,t}$-free graphs, even for $k = 2$.

\begin{proposition}\label{prop:unbounded-tree-alpha}
There exists an infinite family $\mathcal{G}$ of $K_{2,2}$-free $\mathcal{O}_2$-free graphs such that for each $G\in \mathcal{G}$, the tree-independence number of $G$ is at least \hbox{$(\log_2|V(G)|-1)/2$}.
\end{proposition}

\begin{proof}
Bonamy et al.~\cite{MR4723425} constructed an infinite family $\mathcal{G}$ of $\mathcal{O}_2$-free graphs not containing $K_{3,3}$ or $K_3$ as a subgraph such that for each $G\in \mathcal{G}$, the treewidth of $G$ is at least \hbox{$\log_2|V(G)|-1$} (see~\cite[Theorem 2.1]{MR4723425} and the comment following it).
Each such graph $G$ consists of an induced path $P$ such that the set $V(G)\setminus V(P)$ is independent.
Subdividing each edge of $P$ transforms $G$ into a $K_{2,2}$-free $\mathcal{O}_2$-free bipartite graph $G'$ with $|V(G)|\le |V(G')| \le 2|V(G)|$.
Since the treewidth does not change by subdividing edges, we have that $\tw(G') \ge \log_2|V(G)|-1\ge \log_2|V(G')|-2$.
Now, since $G'$ is bipartite, if $G'$ has a tree decomposition $\mathcal{T}$ with independence number $k$, then each bag induces a subgraph of $G'$ with at most $2k$ vertices.
Therefore, $\tw(G')\le 2\tin(G')-1$, which implies that $\tin(G')\ge (\log_2|V(G')|-1)/2$.
\end{proof}

\section{Tree-independence number of $( P_3 + P_1)$-free graphs}\label{sec:P3P1-free}

Recall that for a graph $G$, we denote by $\ibn(G)$ the \emph{induced biclique number} of $G$, that is, the largest nonnegative integer $s$ such that $G$ contains an induced subgraph isomorphic to $K_{s,s}$.
The importance of this parameter for tree-independence number is given by the following lemma.

\begin{lemma}[Dallard et al.~\cite{dallard2024treewidth}]\label{lem:ibn}
Let $G$ be a graph.
Then $\tin (G) \ge \ibn(G)$.
\end{lemma}

Our approach is based on the following structural characterization of $(P_3+P_1)$-graphs.
The complement of $P_3+P_1$ is the \emph{paw}, that is, the graph with vertex set $\{a,b,c,d\}$ such that $\{a,b,c\}$ induce a triangle and there is one additional edge $cd$.
Olariu~\cite{zbMATH04066954} gave the following characterization of paw-free graphs.

\begin{theorem}[Olariu~\cite{zbMATH04066954}]
\label{thm:paw-free}
A graph $G$ is paw-free if and only if each component of~$G$ is triangle-free or complete multipartite.
\end{theorem}

The \emph{join} of two vertex-disjoint graphs $G_1$ and $G_2$ is a graph obtained from the disjoint union of $G_1$ and $G_2$ by adding to it all edges between $G_1$ and $G_2$.

\begin{observation}\label{obs:join}
The independence number of the join of $G_1$ and $G_2$ is $\max\{\alpha(G_1), \alpha(G_2)\}$.
\end{observation}

We are now ready to characterize the tree-independence number of $(P_3 + P_1)$-free graphs.
For simplicity of presentation, we restrict ourselves to graphs with at least one edge.
(Edgeless graphs are chordal and can thus be handled, e.g., with \cref{lem chordal tin}.)

\thmPthreePone*

\begin{proof}
First we show that $\tin(G) \leq \max \{ \ibn(G), 2\}$ by induction on the number $k$ of connected components of $\overline{G}$.
If $k = 1$, then $\overline{G}$ is connected and by~\Cref{thm:paw-free}, it follows that $\overline{G}$ is either triangle-free or complete multipartite.
In the first case, $\alpha(G)\le 2$, and in the second case, $G$ is a disjoint union of complete graphs, hence, $G$ is a chordal graph, and by \cref{lem chordal tin}, $\tin(G)\le 1$.
Hence in both cases, it holds that $\tin(G) \le 2$.

Let $k>1$ and assume that, for any $(P_3+P_1)$-free graph $H$ whose complement has at most $k-1$ connected components, it holds that $\tin(H) \leq \max \{ \ibn(H), 2 \}$.
Let $G$ be a graph whose complement consists of $k$ connected components and let $s=\max\{\ibn(G),2\}$.
Since $\overline{G}$ is not connected, the vertices of $G$ can be partitioned into two nonempty sets $A_1$ and $A_2$, such that $G$ is the join of $G[A_1]$ and $G[A_2]$.

If both $G[A_1]$ and $G[A_2]$ have independence number more that $s$, then $G$ contains $K_{s+1,s+1}$ as induced subgraph, which contradicts the fact that $\ibn(G)\leq s$.
So we may assume without loss of generality that $\alpha(G[A_1])\le s$.
By the induction hypothesis, and since $\ibn(G[A_2])\leq\ibn (G)$, we have that $\tin(G[A_2])\leq s$.
Let $\mathcal{T}$ be a tree decomposition of $G[A_2]$ with independence number at most $s$.
By adding $A_1$ to each bag of $\mathcal{T}$, we get a tree decomposition $\mathcal{T}'$ of $G$.
Then for each bag $X_t'$ of $\mathcal{T}'$, there is a bag $X_t$ of $\mathcal{T}$ such that the subgraph of $G$ induced by $X_t'$ is  the join of $G[X_t]$ and $G[A_1]$, thus, by \Cref{obs:join}, $\alpha(G[X_t'])=\max\{\alpha(G[X_t]),\alpha(G[A_1])\}\le s$.
Therefore, $\tin(G) \le s$.

\begin{sloppypar}
Since $G$ has at least one edge, $\ibn(G)\ge 1$.
Suppose first that $G$ contains an induced $C_4$.
Then, $\ibn(G)\ge 2$ and consequently, by \cref{lem:ibn}, $\ibn(G)\le \tin(G)\le \max \{ \ibn(G), 2\} = \ibn(G)$; hence, equalities must hold throughout, implying in particular that $\tin(G) = \ibn(G)$.
Suppose now that $G$ is $C_4$-free.
Then $\ibn(G) = 1$.
If $G$ contains an induced $C_5$, then ${\tin(G) \ge 2}$ by \cref{lem chordal tin} and $\tin(G)\le \max \{ \ibn(G), 2\} = 2$, implying that $\tin(G) = 2$.
So we may assume that $G$ is also $C_5$-free.
Then $G$ is chordal, since otherwise $G$ would contain an induced cycle of length at least $6$, which is impossible as $G$ is $(P_3+P_1)$-free.
Hence, $1 = \ibn(G) \le \tin(G)\le 1$ by  \cref{lem chordal tin} and equalities must hold throughout; in particular, $\tin(G) = \ibn(G)$.
\end{sloppypar}
\end{proof}

As an immediate consequence of \Cref{thm:P3+P1}, we obtain the following partial support for the case of \Cref{conjecture:excluding-a-path} for the classes of $\{P_5,K_{t,t}\}$-free graphs.

\begin{corollary}\label{cor:P3+P1intro}
Let $t$ be a positive integer and let $G$ be a $\{ P_3 + P_1, K_{t,t}\}$-free graph.
Then, $\tin(G)\le t$.
\end{corollary}

\section{Tree-independence number of $\{P_4+P_1,C_4\}$-free graphs}\label{sec:P4P1C4-free}

In this section, we prove \Cref{thm:tccn}, stating that that the tree-clique-cover number of every $\{P_4+P_1,C_4\}$-free graph is at most $3$.
Our approach is based on first deriving an analogous bound on the tree-independence number, by showing that the class of $\{P_4+P_1,C_4\}$-free graphs is a subclass of the class of $K_{2,3}$-induced-minor-free graphs and applying a known bound on the tree-independence number of $K_{2,3}$-induced-minor-free graphs from~\cite{dallard2023treewidth}.
Then, we show how to refine the bound by invoking a result on the chromatic number of $\{2K_2$, gem$\}$-free graphs from~\cite{zbMATH06994879}.

We start with some definitions.
A \emph{prism} is a graph made of three vertex-disjoint chordless paths $P^1 = a_1\ldots b_1$, $P^2 = a_2\ldots b_2$, and $P^3 = a_3\ldots b_3$ of length at least 1, such that $\{a_1, a_2, a_3\}$ and $\{b_1, b_2, b_3\}$ are triangles and no edges exist between the paths except those of the two triangles.
A prism is \emph{long} if at least one of its three paths has length at least 2.

A \emph{pyramid} is a graph made of three chordless paths $P^1 = a\ldots b_1,$ $P^2 = a\ldots b_2$, $P^3 = a\ldots b_3$ of length at least 1, two of which have length at least 2, vertex-disjoint except at $a$, and such that $\{b_1, b_2, b_3\}$ is a triangle and no edges exist between the paths except those of the triangle and the three edges incident to $a$.

A \emph{theta} is a graph made of three internally vertex-disjoint chordless paths $P^1,P^2,P^3$ of length at least 2 with extremities $a,b$ and such that no edges exist between the paths except the three edges incident to $a$ and the three edges incident to $b$.

A \emph{hole} in a graph is a chordless cycle of length at least 4.
%Observe that the lengths of the paths in the three definitions above are designed so that the union of any two of the paths induces a hole.
A \emph{wheel} $W = (H, x)$ is a graph formed by a hole $H$ (called the \emph{rim}) together with a vertex $x$ (called the \emph{center}) that has at least three neighbors in the hole.
A \emph{sector} of $W$ is a subpath of $H$ between two consecutive neighbors of $x$.
Note that $H$ is edgewise partitioned into the sectors of $W$. Also, every wheel has at least three sectors.
A wheel is \emph{broken} if at least two of its sectors have length at least 2.

The above families appear in the following characterization of induced subgraphs that must be present in any graph containing $K_{2,3}$ as an induced minor, due to Dallard et al.~\cite{DBLP:conf/iwoca/DallardDHMPT24}.

\begin{lemma}[Dallard et al.~\cite{DBLP:conf/iwoca/DallardDHMPT24}]\label{lemma:K_23}
A graph contains $K_{2,3}$ as an induced minor if and only if it contains a long prism, a pyramid, a theta, or a broken wheel as an induced subgraph.
\end{lemma}

Hence, to show that every $\{P_4+P_1,C_4\}$-free graph is $K_{2,3}$-induced-minor-free, it suffices to show the following.

\begin{lemma}\label{lemma:P4+P1-C4-free}
Let $H$ be a long prism, pyramid, theta, or a broken wheel.
Then, $H$ contains either $P_4+P_1$ or $C_4$ as an induced subgraph.
\end{lemma}

\begin{proof}
Let $H$ be a long prism, with triangles $\{ a_1, a_2, a_3\}$ and $\{b_1, b_2, b_3\}$, and vertex-disjoint chordless paths $P^1 = a_1\ldots b_1$, $P^2 = a_2\ldots b_2$, and $P^3 = a_3\ldots b_3$.
If two of the paths of $H$ have length~$1$, then there is an induced $C_4$ in $H$.
Hence, we can assume that w.l.o.g.~that $P^1$ and $P^2$ each have length at least 2.
Let $x$ and $y$ be vertices adjacent to $a_1$ and $a_2$ along $P^1$ and $P^2$, respectively.
Then, the subgraph of $H$ induced by the set $\{x, a_1, a_2, y, b_3\}$ is isomorphic to $P_4+P_1$.
Therefore, $H$ contains either $P_4+P_1$ or $C_4$ as an induced subgraph.

Next, let $H$ be a pyramid, with a triangle $\{b_1, b_2, b_3\}$ and three chordless paths $P^1 = a\ldots b_1,$ $P^2 = a\ldots b_2$, and $P^3 = a\ldots b_3$.
We may assume w.l.o.g.~that $P^1$ and $P^2$ each have length at least 2 and let $x$ and $y$ be vertices adjacent to $a_1$ and $a_2$ along $P^1$ and $P^2$, respectively.
Suppose that $P^1$ has length exactly 2.
If $P^3$ has length 1, then the set $\{ x, b_1, b_3, a \}$ induces a subgraph of $H$ isomorphic to $C_4$.
If $P^3$ has length at least 2, then the set $\{ x, b_1, b_2, y, z \}$ where $z$ is any internal vertex of $P^3$, induces a subgraph of $H$ isomorphic to $P_4+P_1$.
By symmetry, the same happens if $P^2$ has length exactly 2, so we can assume that both $P^1$ and $P^2$ have length at least 3.
Then, vertices $\{ x, b_1, b_2, y, a \}$ induce a subgraph of $H$ isomorphic to $P_4+P_1$.
Hence, $H$ contains either $P_4+P_1$ or $C_4$ as an induced subgraph.

Next, let $H$ be a theta, and let $P^1,P^2 ,P^3$ be its three internally vertex-disjoint chordless paths with extremities $a,b$.
Note that we may assume that no two of these paths each have length exactly $2$, since otherwise there is an induced $C_4$ in $H$.
Suppose w.l.o.g.~that $P^1$ and $P^2$ have length at least 3.
Let $x$, $y$, and $z$ be the vertices adjacent to $a$ along $P^1$, $P^2$ and $P^3$, respectively, let $w$ be the neighbor of $x$ in $P^1$ other than $a$, and let $w'$ be the neighbor of $z$ in $P^3$ other than $a$.
Then, the set $\{ w,x,a,y, w' \}$ induces a subgraph of $H$ isomorphic to $P_4+P_1$.
Therefore, $H$ contains either $P_4+P_1$ or $C_4$ as an induced subgraph.

Finally, let $W=(H,x)$ be a broken wheel, and let $S$ and $S'$ be two sectors of $W$ of length at least 2.
If either $S$ or $S'$ have length exactly 2, then there is an induced $C_4$ in $W$.
So suppose that $S$ and $S'$ have length at least 3.
Then four consecutive vertices of $S$ and an internal vertex of $S'$ not adjacent to the extremities of $S$ induce a subgraph of $W$ isomorphic to $P_4+P_1$.
Hence, $W$ contains either $P_4+P_1$ or $C_4$ as an induced subgraph and this completes the proof.
\end{proof}

\Cref{lemma:P4+P1-C4-free,lemma:K_23} imply the following.

\begin{corollary}\label{P4+P1-C4-free-K23-induced-minor-free}
Every $\{ P_4 + P_1, C_4 \}$-free graph is $K_{2,3}$-induced-minor-free.
\end{corollary}

Let us remark that the result of the above corollary cannot be extended to the class of $\{P_6,C_4\}$-free graphs (in particular, \Cref{conjecture:excluding-a-path} is still open for this case).
In fact, for every integer $t\ge 2$ there exists a $\{P_6,C_4\}$-free graph $G_t$ that is not \hbox{$K_{2,t}$-induced-minor-free}.
Let $G_t$ be the graph obtained from the complete graph with vertex set $\{v_0,v_1,\ldots, v_t\}$ by subdividing once every edge incident with $v_0$.
Then, $G_t$ is a $\{P_6,C_4\}$-free graph that contains $K_{2,t}$ as an induced minor, as witnessed by contracting the edges within the clique $\{v_1,\ldots, v_t\}$.
\medskip

The following upper bound on the tree-independence number of $K_{2,3}$-induced-minor-free graphs is given by Lemmas 3.3 and 3.10 in \cite{dallard2023treewidth}.

\begin{proposition}\label{prop:K_23 bdd tin}
If $G$ is a $K_{2,3}$-induced-minor-free graph, then $\tin (G)\le 3$.
%The tree-independence number of any $K_{2,3}$-induced-minor-free graph is at most 3.
\end{proposition}

\Cref{P4+P1-C4-free-K23-induced-minor-free,prop:K_23 bdd tin} imply
that $\{ P_4+P_1,C_4\}$-free graphs have bounded tree-independence number.

\begin{lemma}\label{lemma:P1+P4-C4-bdd tin}
If $G$ is a $\{ P_4+P_1,C_4\}$-free graph, then $\tin (G)\le 3$.
\end{lemma}

We now show that the above result can be strengthened by replacing the tree-independence number with a larger parameter, obtained by considering, instead of the maximum cardinality of an independent set contained in a bag, the chromatic number of the complement of the subgraph induced by the bag.
A precise definition is as follows.
A \emph{clique cover} of a graph $G$ is a partition of $V(G)$ into cliques.
The \emph{clique cover number} of a graph $G$, denoted by $\theta(G)$, is the minimum number of cliques in a clique cover of $G$.
Since a coloring of a graph $G$ is a partition of $V(G)$ into independent sets, and since $S\subseteq V(G)$ is a clique in $G$ if and only if $S$ is an independent set in $\overline{G}$, it follows that a partition of vertices of $G$ is a clique cover of $G$ if and only if it is a coloring of $\overline{G}$, i.e., $\theta (G) = \chi (\overline{G})$.
The \emph{clique cover number} of a tree decomposition $\mathcal{T}$ of $G$ is defined as
\[\theta(\mathcal{T}) = \max_{t\in V(T)} \theta(G[X_t])\,.\]
The \emph{tree-clique-cover number} of a graph $G$, denoted $\tccn(G)$, is the minimum clique cover number among all possible tree decompositions of $G$ (see Abrishami et al.~\cite{MR4763243}).

Note that a graph $G$ is $\{ P_4+P_1,C_4\}$-free if and only if its complement $\overline{G}$ is $\{2K_2$, gem$\}$-free, where the \emph{gem} is the join of $P_4$ and $P_1$.
Brause et al.~\cite[Corollary 18]{zbMATH06994879} gave the following bound on the chromatic number of $\{2K_2$, gem$\}$-free graphs.

\begin{lemma}[Brause et al.~\cite{zbMATH06994879}]\label{lemma:2K_2-gem}
If $G$ is a $\{ 2K_2$, gem$\}$-free graph, then $\chi(G) \le \max \{\omega(G), 3\}$.
\end{lemma}

Based on these observations, we are ready to prove the main result of this section.

\thmtccn*

\begin{proof}
Let $G$ be a $\{ P_4+P_1,C_4\}$-free graph.
By \Cref{lemma:P1+P4-C4-bdd tin}, there is a tree decomposition $\mathcal{T}=(T, \{ X_t \}_{t\in V(T)})$ of $G$ with independence number at most 3.
Fix $t\in V(T)$.
Then, by \Cref{lemma:2K_2-gem}, \[\theta(G[X_t]) = \chi (\overline{G[X_t]})\le \max \{\omega (\overline{G[X_t]}), 3 \} = \max \{ \alpha(G[X_t]), 3 \}\le 3\,.\]
Hence, $\tccn(G)\le 3$, as claimed.
\end{proof}

\section{Conclusion}\label{sec:conclusion}

We conclude the paper with a few questions left open by our work.

First, let us note a conjecture due to Gartland and Lokshtanov whose validity would imply \Cref{conj:planar-H-no-star}.
Given a graph $G$, a set $S\subseteq V(G)$ is said to be a \emph{balanced separator} in $G$ if all components $C$ of the graph $G-S$ satisfy $|V(C)|\le|V(G)|/2$.
Given a graph $G$ and two sets $X,Y\subseteq V(G)$, we say that $X$ \emph{dominates} $Y$ if $Y\subseteq N[X]$.

\begin{conjecture}[Induced Grid Minor Conjecture,~\cite{Gartland23}]\label{conj:Gartland}
There exists a function $f$ such that for every planar graph $H$, every $H$-induced-minor-free graph $G$ has a balanced separator dominated by $f(H)$ vertices.
\end{conjecture}

In~\cite{Gartland23}, the conjecture is stated only for the case when $H$ is a $(k\times k)$-grid; however, the two conjectures are equivalent, since every planar graph is an induced minor of a sufficiently large grid graph (see~\cite[Theorem 12]{campbell2024treewidthhadwigernumberinduced}).

Observe that in a $K_{1,t}$-free graph $G$, a set $S\subseteq V(G)$ that is dominated by a set of at most~$k$ vertices induces a subgraph with independence number at most $kt$.
Hence, by \cite[Lemma 7.1]{arXiv:2405.00265} (see also~\cite{dallard_et_al:LIPIcs.ICALP.2024.51}),
\Cref{conj:Gartland} implies the following strengthening of \Cref{conj:planar-H-no-star}.

\begin{sloppypar}
\begin{conjecture}\label{conj:linear}
For every planar graph $H$, there exists a constant $c_H$ such that for every positive integer $t$, every $K_{1,t}$-free $H$-induced-minor-free graph $G$ satisfies \hbox{$\tin(G)\le c_H\cdot t$}.
\end{conjecture}
\end{sloppypar}

While a positive resolution to \Cref{conj:linear} would prove \Cref{conj:planar-H-no-star}, a negative one would disprove \Cref{conj:Gartland}.
Similar properties holds for a variant of~\Cref{conj:linear} asking for an upper bound on the tree-independence number with a polynomial, instead of linear, dependency on $t$.
The existing proofs of special cases of \Cref{conj:planar-H-no-star} give bounds on tree-independence number that are linear in $t$ if $H=kC_3$ for some positive integer $k$ (both by the proof of Ahn et al.~\cite{ahn2024coarseerdhosposatheorem} as well as our proof) or if $H$ is a cycle (from~\cite[Theorem 3.3]{MR3425243}), polynomial in $t$ (more precisely, in $\mathcal{O}_H(t^4)$) if $H\in \mathcal{S}$~\cite{dallard2024treewidth}, and exponential in $t$ if $H$ is a wheel~\cite{CHMW2025}.

\medskip
\Cref{prop:unbounded-tree-alpha} implies that for any two integers $k\ge 2$ and $t\ge 2$, the tree-independence number of $K_{t,t}$-free $\mathcal{O}_k$-free graphs can be logarithmic in the number of vertices.
We wonder whether this is the worst that can happen.

\begin{question}\label{q1}
Is it true that for every two positive integers $k$ and $t$ there exists a constant $c_{k,t}>0$ such that if $G$ is a $K_{t,t}$-free $\mathcal{O}_k$-free graph with at least two vertices, then $\tin(G)\le c_{k,t}\cdot \log |V(G)|$?
\end{question}

\begin{sloppypar}
As a related and possibly easier variant of \cref{q1}, we could ask about a \hbox{polylogarithmic} bound on the tree-independence number of $K_{t,t}$-free $\mathcal{O}_k$-free graphs.
Such bounds were recently established for a number of graph classes (see~\cite{DBLP:conf/soda/ChudnovskyGHLS25,arXiv:2405.00265,chudnovsky2025treeindependencenumberv}).
\end{sloppypar}

\medskip
As shown by~\cite[Proposition 6.3]{dallard2024treewidth} and \Cref{thm:P3+P1}, respectively, for every graph that is either $P_4$-free or $(P_3+P_1)$-free, the tree-independence number differs from the induced biclique number by at most a constant.
We ask whether this is true, more generally, in the class of $P_5$-free graphs.

\begin{question}\label{q2}
Is there a constant $c$ such that $\tin(G)\le \ibn(G)+c$ for every $P_5$-free graph~$G$?
\end{question}

An affirmative answer to \cref{q2} would settle \Cref{conjecture:excluding-a-path} for $k = 5$.
As a possibly simpler variant of the question, one could consider it first for the subclass of $2P_2$-free graphs.

\medskip
By \Cref{thm:tccn}, every $\{P_4+P_1,C_4\}$-free graph admits a tree decomposition such that every bag is a union of three cliques.
We do not know if this bound is sharp: while the $5$-cycle is an example of a $\{P_4+P_1,C_4\}$-free graph with tree-clique-cover number equal to~$2$, we do not know whether this bound must ever be exceeded.
As a possible approach to this question, we propose the following.

\begin{question}\label{q3}
Does every $\{P_4+P_1,C_4\}$-free graph $G$ have a clique $K$ such that $G-K$ is a chordal graph?
\end{question}

\paragraph{Acknowledgements.}

The authors are grateful to Kenny Bešter Štorgel, Cl\'ement Dallard, Vadim Lozin, and Viktor Zamaraev for helpful discussions.
This work is supported in part by the Slovenian Research and Innovation Agency (I0-0035, research program P1-0285 and research projects J1-3003, J1-4008, J1-4084, J1-60012, and N1-0370) and by the research program CogniCom (0013103) at the University of Primorska.

\bibliographystyle{abbrvurl}

\bibliography{biblio}

\end{document}